\documentclass[12pt,reqno]{amsart}

\usepackage{amssymb}

\usepackage[all]{xy}

\numberwithin{equation}{section}

\newtheorem{theorem}{Theorem}[section]
\newtheorem{proposition}[theorem]{Proposition}
\newtheorem{lemma}[theorem]{Lemma}

\theoremstyle{definition}
\newtheorem{definition}[theorem]{Definition}
\newtheorem{remark}[theorem]{Remark}

\begin{document}

\baselineskip=15.5pt

\title[Genuinely ramified maps and pseudo-stable bundles]{Genuinely
ramified maps and pseudo-stable vector bundles}

\author[I. Biswas]{Indranil Biswas}

\address{School of Mathematics, Tata Institute of Fundamental
Research, Homi Bhabha Road, Mumbai 400005, India}

\email{indranil@math.tifr.res.in}

\author[A. J. Parameswaran]{A. J. Parameswaran}

\address{School of Mathematics, Tata Institute of Fundamental
Research, Homi Bhabha Road, Mumbai 400005, India}

\email{param@math.tifr.res.in}



\date{}

\begin{abstract}
Let $X$ and $Y$ be irreducible normal projective varieties, of same dimension, defined
over an algebraically closed field, and let $f\,\, :\,\, Y\, \longrightarrow\, X$ be
a finite generically smooth morphism such that the corresponding homomorphism between the \'etale
fundamental groups $f_*\,:\,\pi^{\rm et}_{1}(Y) \, \longrightarrow\,\pi^{\rm et}_{1}(X)$
is surjective. Fix a polarization on $X$ and equip $Y$ with the pulled back
polarization. For a point $y_0\,\in\, Y$, let $\varpi(Y,\, y_0)$ (respectively,
$\varpi(X,\, f(y_0))$) be the affine group scheme given by the neutral Tannakian category
defined by the strongly pseudo-stable
vector bundles of degree zero on $Y$ (respectively, $X$). We prove that the homomorphism
$\varpi(Y,\, y_0)\, \longrightarrow\, \varpi(X,\, f(y_0))$ induced by $f$ is surjective.
Let $E$ be a pseudo-stable vector bundle on $X$ and $F\, \subset\, f^*E$ a pseudo-stable
subbundle with $\mu(F)\,=\, \mu(f^*E)$. We prove that $f^*E$ is pseudo-stable and there is
a pseudo-stable subbundle $W\, \subset\, E$ such that $f^*W\,=\, F$ as subbundles of $f^*E$.
\end{abstract}

\maketitle

\section{Introduction}

Let $X$ and $Y$ be irreducible normal projective varieties, of same dimension, defined
over an algebraically closed field $k$. Let
$$
f\,\, :\,\, Y\, \longrightarrow\, X
$$
be a morphism such that
\begin{itemize}
\item $f$ is finite,

\item $f$ is generically smooth, and

\item the homomorphism between the \'etale fundamental groups induced by $f$
$$
f_*\,:\,\pi^{\rm et}_{1}(Y) \, \longrightarrow\,\pi^{\rm et}_{1}(X)
$$
is surjective.
\end{itemize}
Fix an ample line bundle (also called a polarization) $L$ on $X$ and equip $Y$ with the polarization $f^*L$.

For any stable vector bundle $E$ on $X$, it is known that the pulled back vector bundle $f^*E$ is stable
\cite[Theorem 1.2]{BDP}. Here we consider the pullbacks of the more general class of pseudo-stable
vector bundles on $X$. A semistable vector bundle is called a pseudo-stable vector bundle if
it admits a filtration of subbundles
such that each successive quotient is a stable vector bundle whose slope coincides with the slope of the
ambient bundle.
A pseudo-stable vector bundle is called strongly pseudo-stable if its pullbacks by the powers of the
absolute Frobenius morphism are also pseudo-stable. (These are recalled in
Definition \ref{def1}; see also Remark \ref{rem-a} and Remark \ref{rb}.) The following theorem is proved here (see Theorem \ref{thm1}):

\begin{theorem}\label{thm0}
Let $E$ be a pseudo-stable vector bundle on $X$. Then the following statements hold:
\begin{enumerate}
\item The pullback $f^*E$ is pseudo-stable.

\item Let $F\, \subset\, f^*E$ be any pseudo-stable subbundle such that
$\mu(F)\,=\, \mu(f^*E)$. Then there is a pseudo-stable subbundle $V\,
\subset\, E$ such that the subbundle
$$
f^*V\, \subset\, f^*E
$$
coincides with the subbundle $F\, \subset\, f^*E$.
\end{enumerate}
\end{theorem}

Let $Z$ be a normal projective variety defined over $k$. Fix a point $z_0\, \in\, Z$,
and also fix an ample line bundle on $Z$. 
Consider the category of strongly pseudo-stable vector bundles on $Z$ of degree zero.
It is a neutral Tannakian category \cite[Theorem 1]{BaPa}. We denote by $\varpi(Z,\, z_0)$
the proalgebraic group scheme over $k$ defined by this neutral Tannakian category. This
neutral Tannakian category given by the strongly pseudo-stable vector bundles on $Z$ of degree zero
has a full subcategory given by all strongly pseudo-stable vector bundle $E$ on $Z$ such that
$c_i(E)$ is numerically equivalent to zero for all $i\, \geq\, 1$. This subcategory is actually
independent of the choice of the polarization on $Z$. Moreover, this subcategory is again a neutral Tannakian
category. We denote by $\varpi^S(Z,\, z_0)$ proalgebraic group scheme over $k$ defined by this
neutral Tannakian category.

As an application of Theorem \ref{thm0}, we prove the following (see Theorem \ref{thm3a}):

\begin{theorem}\label{thm02}
Fix a point $y_0\, \in\, Y$ and an ample line bundle on $X$; equip $Y$ with the pulled
back polarization. Then the homomorphism
$$
f_*\,\, :\,\, \varpi(Y,\, y_0)\, \longrightarrow\, \varpi(X,\, f(y_0))
$$
induced by $f$ is faithfully flat (in other words, the homomorphism $f_*$ is surjective).
\end{theorem}

Let $X$ and $Y$ be irreducible normal projective varieties (we do \textit{not} assume that
$\dim X\,=\, \dim Y$), and let
$$
f\,\, :\,\, Y\, \longrightarrow\, X
$$
be a morphism such that
\begin{itemize}
\item $f$ is surjective,

\item $f$ is generically smooth, and

\item the homomorphism between the \'etale fundamental groups induced by $f$
$$
f_*\,:\,\pi^{\rm et}_{1}(Y) \, \longrightarrow\,\pi^{\rm et}_{1}(X)
$$
is surjective.
\end{itemize}

We prove the following (see Proposition \ref{prop3}):

\begin{proposition}\label{prop0}
Fix a point $y_0\, \in\, Y$. Then the homomorphism
$$
f_*\,\, :\,\, \varpi^S(Y,\, y_0)\, \longrightarrow\, \varpi^S(X,\, f(y_0))
$$
induced by $f$ is faithfully flat.
\end{proposition}

\section{Direct image of semistable bundles}

The base field $k$ is assumed to be algebraically closed; there is no assumption
on its characteristic.

Take an irreducible normal projective variety $X$ defined over $k$. Fix an ample line bundle $L$ on $X$;
using it, define the degree
$$
{\rm degree}(F)\, \in\,\mathbb Z
$$
of any torsionfree coherent $F$ sheaf on $X$ (see \cite[p.~13--14, Definition 1.2.11]{HL}
for the details). The \textit{slope} $\mu (F)$ of $F$ is defined to be
$$
\mu(F)\, :=\, \frac{{\rm degree}(F)}{{\rm rank}(F)}\, .
$$
The sheaf $F$ is called \textit{stable} (respectively, \textit{semistable}) if
$$
\mu(V)\, <\, \mu(F) \ \ \text{(respectively,}\ \mu(V)\, \leq\, \mu(F)\text{)}
$$
for all coherent subsheaves $V\, \subset\, F$ with $0\, <\, {\rm rank}(V)\, <\, {\rm rank}(F)$
\cite[p.~14, Definition 1.2.11]{HL}. Also, $F$ is called \textit{polystable} if it decomposes into
a direct sum of stable sheaves of same slope.

When the characteristic of the field $k$ is positive, for any variety $M$ over $k$, let
$$
F_M\,:\, M\, \longrightarrow\, M
$$
denote the absolute Frobenius morphism of $M$.

\begin{definition}\label{def1}\mbox{}
A semistable vector bundle $F$ is called a \textit{pseudo-stable vector bundle} if
$F$ admits a filtration of subbundles
$$
0\,=\, F_0\, \subsetneq\, F_1\, \subsetneq\, F_2\, \subsetneq\, \cdots \, \subsetneq\, F_{n-1}
\, \subsetneq\, F_n\,=\, F
$$
such that $F_i/F_{i-1}$ is a stable vector bundle with $\mu(F_i/F_{i-1})\,=\, \mu(F)$ for every
$1\, \leq\, i\, \leq \, n$.

A pseudo-stable vector bundle $F$ is called \textit{strongly pseudo-stable} if the vector bundle $(F^n_Z)^*F$ is
pseudo-stable for all $n\, \geq\, 1$, where $F_Z$ is the absolute Frobenius morphism of $Z$.
\end{definition}

\begin{remark}\label{rem-a}
Regarding Definition \ref{def1}, it should be mentioned that the category of semistable vector bundles on $X$ of fixed slope
is not an abelian category if $\dim X \, >\, 1$. For example, when $X$ is a surface, and $Z\,\subset\, X$ is a subscheme of
dimension zero, then there is a short exact sequence
$$
0\, \longrightarrow\, {\mathcal O}_X \, \longrightarrow\, E \, \longrightarrow\, {\mathcal I}_Z \, \longrightarrow\, 0,
$$
where ${\mathcal I}_Z\, \subset\, {\mathcal O}_X$ is the ideal sheaf of $Z$, and $E$ is a vector bundle of rank two. Note
that $E$ is semistable of degree zero.
\end{remark}

\begin{remark}\label{rb}
It should be clarified that strongly pseudo-stable bundles are precisely the lf-graded bundles in \cite{BaPa}. The terminology
``pseudo-stable'' was introduced in \cite{BG} before \cite{BaPa}.
\end{remark}

For a torsionfree sheaf $F$ on $X$, let $F_1$ be the first nonzero term of the Harder--Narasimhan filtration
of $F$; this $F_1$ is called the \textit{maximal semistable subsheaf} of $F$ \cite[p.~16, Definition 1.3.6]{HL}.
Note that when $F$ is semistable, the maximal semistable subsheaf of $F$ is $F$ itself. Define
$$
\mu_{\rm max}(F)\,:=\, \mu(F_1)\, .
$$

Let $X$ and $Y$ be irreducible normal projective varieties of common dimension $d$, and let
\begin{equation}\label{e1}
f\,\, :\,\, Y\, \longrightarrow\, X
\end{equation}
be a morphism such that
\begin{itemize}
\item $f$ is finite,

\item $f$ is generically smooth, and

\item the homomorphism between the \'etale fundamental groups induced by $f$
$$
f_*\,:\,\pi^{\rm et}_{1}(Y) \, \longrightarrow\,\pi^{\rm et}_{1}(X)
$$
is surjective.
\end{itemize}
Such a morphism $f$ is called a genuinely ramified map (see \cite{BP1}, \cite{BDP}).

Fix an ample line bundle $L$ on $X$. The pulled back line bundle $f^*L$ on $Y$ is ample, because
$f$ is a finite morphism. We will define the degree of torsionfree coherent sheaves on $Y$ (respectively, $X$)
using $f^*L$ (respectively, $L$).

We recall a lemma from \cite{BDP}.

\begin{lemma}[{\cite[Lemma 3.1]{BDP}}]\label{lem1}
For any semistable vector bundle $E$ on $Y$,
$$
\mu_{\rm max}(f_*E)\,\, \leq\,\, \frac{1}{{\rm degree}(f)}\cdot\mu(E)\, .
$$
\end{lemma}

Let $E$ be a pseudo-stable vector bundle on $Y$.
Assume that $f_*E$ contains a polystable locally free subsheaf $E'$ such that
\begin{equation}\label{f2}
\mu(E')\,=\, \frac{1}{{\rm degree}(f)}\cdot\mu(E)\, .
\end{equation}
Let $E_1\, \subset\, f_*E$ be the maximal semistable subsheaf, so we have
$\mu(E_1)\,=\, \mu_{\rm max}(f_*E)$ and $E_1$ is the maximal subsheaf of $f_*E$
satisfying this equality. From Lemma \ref{lem1} and \eqref{f2} it follows that
$$
\mu_{\rm max}(f_*E)\, =\, \mu(E_1)\,=\, \frac{1}{{\rm degree}(f)}\cdot\mu(E)\, ,
$$
and $E'\, \subset\, E_1$. From
\cite[Theorem 1.2 (b)]{BP2} we know that there is a unique coherent subsheaf
\begin{equation}\label{j1}
W\, \subset\, E_1
\end{equation}
satisfying the following four conditions:
\begin{enumerate}
\item[(1)] $\mu(W)\,=\, \mu(E_1)$,

\item[(2)] $W$ is a pseudo-stable vector bundle,

\item[(3)] $E_1/W$ is torsionfree, and

\item[(4)] $W$ is maximal among all subsheaves of $E_1$ satisfying the above three conditions.
\end{enumerate}
In particular, $W$ in \eqref{j1} is locally free.

\textbf{Notation.}\, We will call the subsheaf $W$ in \eqref{j1} the \textit{pseudo-stable socle}
of $E$. This was suggested by the referee. 

\begin{proposition}\label{prop1}
Take $E$ as above. For the pseudo-stable socle $W\, \subset\, E_1$ in \eqref{j1},
$$
{\rm rank}(W)\,\,\leq\,\, {\rm rank}(E)\, .
$$
\end{proposition}

\begin{proof}
First assume that $E$ is stable. Then from \cite[Theorem 3.2]{BDP} we know that
there is a locally free subsheaf $V\, \subset\, f_*E$ such that the natural map
$f^*V\,\longrightarrow\, E$ is an isomorphism; the homomorphism
$f^*V\,\longrightarrow\, E$ is given by the inclusion map $V\, \hookrightarrow\, f_*E$
using the natural isomorphism
$$
H^0(Y,\, \text{Hom}(f^*V,\, E))\,=\, H^0(X,\, \text{Hom}(V,\, f_*E))
$$
(see \cite[p.~110]{Ha}). Note that $V$ is stable with
$\mu(V)\,=\, \frac{1}{{\rm degree}(f)}\cdot\mu(E)$ (see the first part of the
proof of \cite[Theorem 3.2]{BDP}). Consequently, we have
\begin{equation}\label{e5}
V\, \subset\, W\, ,
\end{equation}
where $W$ is the pseudo-stable socle in \eqref{j1}.

Since $f^*V\,=\, E$, from the projection formula it follows that
\begin{equation}\label{e4}
f_*E\,=\, f_*f^*V\,=\, V\otimes f_*{\mathcal O}_Y\, .
\end{equation}
We have an inclusion map
\begin{equation}\label{e6}
{\mathcal O}_X\, \hookrightarrow\, f_*{\mathcal O}_Y\, ,
\end{equation}
given by the identity map of ${\mathcal O}_Y$ using the isomorphisms
$$
H^0(X,\, {\rm Hom}({\mathcal O}_X,\, f_*{\mathcal O}_Y))\,=\,
H^0(Y,\, {\rm Hom}(f^*{\mathcal O}_X,\,
{\mathcal O}_Y))\,=\, H^0(Y,\, {\rm Hom}({\mathcal O}_Y,\,{\mathcal O}_Y))
$$
(see \cite[p.~110]{Ha}). From \eqref{e6} and \eqref{e4} it follows that
$V\, \subset\, V\otimes f_*{\mathcal O}_Y\, =\, f_*E$, and consequently,
\begin{equation}\label{e7}
(f_*E)/V\,=\, V\otimes ((f_*{\mathcal O}_Y)/{\mathcal O}_X)\, .
\end{equation}
{}From the proof of \cite[Lemma 4.1]{BP1} we know that
$$
H^0(X,\, \text{Hom}(\mathcal{W},\, V\otimes (f_*{\mathcal O}_Y)/{\mathcal O}_X))\, =\, 0
$$
for any pseudo-stable vector bundle $\mathcal W$ on $X$ with $\mu(\mathcal{W})\,=\, \mu(V)$
(see also \cite[Theorem 2.5]{BDP}).
Therefore, we have $W\, \subset\, V$. This and \eqref{e5} together imply that $V\,=\, W$;
note that any homomorphism between two vector bundles of same degree, which
is generically an isomorphism, is in fact an isomorphism. Since $f^*V\,=\, E$, we conclude that
$$
{\rm rank}(W)\,\,=\,\, {\rm rank}(V)\,\,=\,\, {\rm rank}(E)\, .
$$

Now take the vector bundle $E$ to be pseudo-stable (as in the statement of the proposition).
Let
\begin{equation}\label{e11}
0\,=\, F_0\, \subsetneq\, F_1\, \subsetneq\, F_2\, \subsetneq\, \cdots \, \subsetneq\, F_{n-1}
\, \subsetneq\, F_n\,=\, E
\end{equation}
be a filtration of subbundles
such that $F_i/F_{i-1}$ is a stable vector bundle with $\mu(F_i/F_{i-1})\,=\, \mu(E)$ for every
$1\, \leq\, i\, \leq \, n$ (see Definition \ref{def1}). Let
$$
0\,=\, f_*F_0\, \subsetneq\, f_*F_1\, \subsetneq\, f_*F_2\, \subsetneq\, \cdots \,
\subsetneq\, f_*F_{n-1}\, \subsetneq\, f_*F_n\,=\, f_*E
$$
be the corresponding filtration of $f_*E$. Take any $1\, \leq\, i\, \leq \, n$, and
consider the vector bundle $(f_*F_i)\big/(f_*F_{i-1})$. Since $f$ is a finite map, the higher
direct images vanish, and hence $(f_*F_i)\big/(f_*F_{i-1})\,=\, f_*(F_i/F_{i-1})$. From the
previous case of stable $E$ we know that one
of the following two statements holds for this $f_*(F_i/F_{i-1})$:
\begin{enumerate}
\item The vector bundle $f_*(F_i/F_{i-1})$ contains a polystable locally free subsheaf
$E'_i$ such that
$$
\mu(E'_i)\,=\, \frac{1}{{\rm degree}(f)}\cdot\mu(F_i/F_{i-1})\,=\,
\frac{1}{{\rm degree}(f)}\cdot\mu(E)\, .
$$
Then there is a stable locally free subsheaf $W_i\, \subset\, f_*(F_i/F_{i-1})$ such that
$f^*W_i\,=\, F_i/F_{i-1}$.
Furthermore,
$$
H^0(X,\, \text{Hom}({\mathcal E},\, f_*(F_i/F_{i-1})/W_i))\, =\, 0
$$
for any pseudo-stable vector bundle $\mathcal E$ on $X$ with $\mu({\mathcal E})\,=\, 
\mu(W_i)\,=\, \frac{1}{{\rm degree}(f)}\cdot\mu(E)$.

\item For every polystable locally free subsheaf
$\mathcal F$ of $f_*(F_i/F_{i-1})$,
$$
\mu({\mathcal F})\, < \, \frac{1}{{\rm degree}(f)}\cdot\mu(F_i/F_{i-1})
\,=\, \frac{1}{{\rm degree}(f)}\cdot\mu(E)\, .
$$
\end{enumerate}
If the second alternative holds for $f_*(F_i/F_{i-1})$, then set $W_i\,=\, 0$.

Therefore, we have
\begin{equation}\label{e8a}
{\rm rank}(W_i)\, \leq\, {\rm rank}(F_i/F_{i-1})\, .
\end{equation}
From the properties of $W_i$ it follows that
\begin{equation}\label{e8}
{\rm rank}(W)\, \,\leq\,\, \sum_{i=1}^n {\rm rank}(W_i)\, .
\end{equation}
Indeed, the quotient $(W\bigcap F_i)\big/(W\bigcap F_{i-1})$ has a natural injective
homomorphism to $W_i$; this evidently implies \eqref{e8}.
On the other hand,
\begin{equation}\label{e9}
\sum_{i=1}^n {\rm rank}(W_i)
\, \,\leq\,\, \sum_{i=1}^n {\rm rank}(F_i/F_{i-1})\,=\, {\rm rank}(E)
\end{equation}
(see \eqref{e8a}). The proposition follows from \eqref{e8} and \eqref{e9}.
\end{proof}

The following proposition gives a sufficient condition for a pseudo-stable vector bundle
on $Y$ to be the pullback of a vector bundle on $X$.

\begin{proposition}\label{prop2}
Take $E$ as in Proposition \ref{prop1}. Assume that
${\rm rank}(W)\,=\, {\rm rank}(E)$, where $W$ is the pseudo-stable socle in \eqref{j1}. Then
$E\,=\, f^* W$.
\end{proposition}

\begin{proof}
The inclusion map $W\, \hookrightarrow\, f_*E$ produces a homomorphism
\begin{equation}\label{e12}
\phi\,\, :\,\, f^*W\, \longrightarrow\, E
\end{equation}
using the natural isomorphism
$$
H^0(Y,\, \text{Hom}(f^*W,\,E))\,=\, H^0(X,\, \text{Hom}(W,\,f_*E))
$$
(see \cite[p.~110]{Ha}).

Since ${\rm rank}(W)\,=\, {\rm rank}(E)$, the inequalities in 
\eqref{e8} and \eqref{e9} are actually equalities. In particular,
$$
{\rm rank}(W_i) \, \,=\,\, {\rm rank}(F_i/F_{i-1})
$$
for all $1\,\leq\, i\, \leq\, n$ (see \eqref{e9}). As mentioned before, the vector bundle
$W$ has a filtration $\{W\bigcap f_*F_i\}_{i=1}^n$ of subbundles such that graded bundle is
a subsheaf of $\bigoplus_{i=1}^n W_i$. Since ${\rm rank}(W)\,=\, \text{rank}(E)\,=\,
\sum_{i=1}^n \text{rank}(F_i/F_{i-1})$, it follows that $(W\bigcap f_*F_i)\big/(W\bigcap f_*F_{i-1})
\,=\, W_i$; here we are using the fact that any homomorphism between two vector bundles of same
degree is actually an isomorphism if it is generically an isomorphism. This filtration
$\{W\bigcap f_*F_i\}_{i=1}^n$ of $W$ pulls back to a filtration
$\{f^*(W\bigcap f_*F_i)\}_{i=1}^n$ of
$f^*W$; the corresponding graded bundle is of course $\bigoplus_{i=1}^n f^*W_i$.
Consider the filtration $\{F_i\}_{i=0}^n$ of $E$ in \eqref{e11}.
The homomorphism $\phi$ in \eqref{e12} is a homomorphism of filtered bundles, so
it produces a homomorphism
\begin{equation}\label{e13}
\phi_i\,\, :\,\, f^*W_i\, \longrightarrow\, F_i/F_{i-1}
\end{equation} 
for every $1\,\leq\, i\, \leq\, n$.
Since $F_i/F_{i-1}$ is stable, and ${\rm rank}(W_i) \, \,=\,\, {\rm rank}(F_i/F_{i-1})$,
it follows that $\phi_i$ in \eqref{e13} is an isomorphism for all
$1\,\leq\, i\, \leq\, n$ (see the first part of the proof of Proposition \ref{prop1}).

Since each $\phi_i$ in \eqref{e13} is an isomorphism, we conclude that
$\phi$ in \eqref{e12} is an isomorphism.
\end{proof}

\section{Descent of pseudo-stable subbundles}

\begin{theorem}\label{thm1}
Consider the map $f$ in \eqref{e1}. Let $E$ be a pseudo-stable vector bundle on $X$. The
following statements hold:
\begin{enumerate}
\item The pullback $f^*E$ is pseudo-stable.

\item Let $F\, \subset\, f^*E$ be any pseudo-stable subbundle such that $\mu(F)
\,=\, \mu(f^*E)$. Then there is a pseudo-stable subbundle $V\, \subset\, E$ such that
the subbundle
$$
f^*V\, \subset\, f^*E
$$
coincides with the subbundle $F\, \subset\, f^*E$.
\end{enumerate}
\end{theorem}

\begin{proof}
Let
$$
0\,=\, E_0\, \subsetneq\, E_1\, \subsetneq\, E_2\, \subsetneq\, \cdots \, \subsetneq\, E_{n-1}
\, \subsetneq\, E_n\,=\, E
$$
be a filtration of subbundles such that $E_i/E_{i-1}$ is a stable vector bundle with
$\mu(E_i/E_{i-1})\,=\, \mu(E)$ for every
$1\, \leq\, i\, \leq \, n$ (see Definition \ref{def1}). Consider the filtration of
subbundles
$$
0\,=\, f^*E_0\, \subsetneq\, f^*E_1\, \subsetneq\, f^*E_2\, \subsetneq\, \cdots \,
\subsetneq\, f^*E_{n-1}\, \subsetneq\, f^*E_n\,=\, f^*E\, .
$$
For every $1\, \leq\, i\, \leq \, n$, the quotient bundle $f^*(E_i/E_{i-1})\,=\
(f^*E_i)\big/(f^*E_{i-1})$ is stable because $E_i/E_{i-1}$ is so \cite[Theorem 1.2]{BDP}.
Also, we have
$$
\mu((f^*E_i)\big/(f^*E_{i-1}))\,=\, \text{degree}(f)\cdot \mu(E_i/E_{i-1})\,=\,
\text{degree}(f)\cdot \mu(E)\,=\, \mu(f^*E)\, .
$$
Consequently, the vector bundle $f^*E$ is pseudo-stable. This proves the first statement.

To prove the second statement, take any pseudo-stable subbundle
$$F\, \subset\, f^*E$$
such that $\mu(F)\,=\, \mu(f^*E)$. So $(f^*E)/F$ is also
pseudo-stable with $\mu((f^*E)/F)\,=\, \mu(f^*E)$. Therefore, 
we have the short exact sequence of
pseudo-stable bundles of same slope
$$
0\, \longrightarrow\, F \, \longrightarrow\,f^*E \, \longrightarrow\, (f^*E)/F
\, \longrightarrow\, 0\, .
$$
It produces the short exact sequence of vector bundles
\begin{equation}\label{e14}
0\, \longrightarrow\, f_*F \, \longrightarrow\,f_*f^*E\,=\, E\otimes f_*{\mathcal O}_Y
\, \longrightarrow\, f_*((f^*E)/F)\,=\, (E\otimes f_*{\mathcal O}_Y)/(f_*F)
\, \longrightarrow\, 0
\end{equation}
on $X$; the higher direct images vanish because the map
$f$ is finite. We will construct pseudo-stable locally free
subsheaves $V_1$, $V_2$ and $V_3$ of $f_*F$, $f_*f^*E$
and $f_*((f^*E)/F)$ respectively.

Let $B\, \longrightarrow\, Y$ be a pseudo-stable vector bundle. 
First assume that $f_*B$ contains a polystable locally free subsheaf $B'$ such that
$$
\mu(B')\,\,=\,\, \frac{1}{{\rm degree}(f)}\cdot\mu(B)\, .
$$
Let $M_B\, \subset\, f_*B$ be the unique maximal semistable subsheaf.
So from Lemma \ref{lem1} it follows that
$$
\mu_{\rm max}(f_*B)\, =\, \mu(M_B)\,=\, \frac{1}{{\rm degree}(f)}\cdot\mu(B)\, .
$$
From \cite[Theorem 1.2]{BP2} we know that there is a unique coherent subsheaf
$$
W_B\, \subset\, M_B
$$
satisfying the following four conditions:
\begin{enumerate}
\item[(1)] $\mu(W_B)\,=\, \mu(M_B)$,

\item[(2)] $W_B$ is a pseudo-stable vector bundle,

\item[(3)] $M_B/W_B$ is torsionfree, and

\item[(4)] $W_B$ is maximal among all subsheaves of $M_B$ satisfying the above three conditions.
\end{enumerate}

Now set $V_B\,=\, W_B$.

If
$$
\mu(B')\,<\, \frac{1}{{\rm degree}(f)}\cdot\mu(B)
$$
for all polystable locally free subsheaves $B'\, \subset\, f_*B$, then set
$V_B\,=\, 0$.

Substituting $F$, $f^*E$ and $(f^*E)/F$ in place of $B$ in the above construction of
$V_B$, we get locally free pseudo-stable subsheaves $$V_1\,:=\, V_F,\ V_2\,:=\, V_{f^*E}\ \text{ and }\
V_3\,:=\, V_{(f^*E)/F}$$ of $f_*F$, $f_*f^*E$
and $f_*((f^*E)/F)$ respectively. Note that we have an exact sequence
\begin{equation}\label{e15}
0\, \longrightarrow\, V_1 \, \longrightarrow\, V_2 \, \longrightarrow\, V_3\, .
\end{equation}
It should be clarified that \eqref{e15} does not assert the surjectivity of the map
$V_2 \, \longrightarrow\, V_3$.

Since $E\, \hookrightarrow\, E\otimes f_*{\mathcal O}_Y\,=\, f_*f^*E$ (see \eqref{e6}), we
have $E\, \subset\, V_2$. Hence from Proposition \ref{prop1} it follows that
${\rm rank}(E)\,=\, {\rm rank}(V_2)$. In view of this equality, from \eqref{e15} and
Proposition \ref{prop1} we conclude that
$$
{\rm rank}(V_1)\,=\, {\rm rank}(F)\,\ \ \text{ and }\,\ \
{\rm rank}(V_3)\,=\, {\rm rank}((f^*E)/F)\, .
$$
Now from Proposition \ref{prop2} it follows that $f^*V_1\,=\, F$ as subbundles of
$f^*E$.
\end{proof}

\begin{remark}\label{rem2}
If $W$ is a pseudo-stable vector bundle, and $F\, \subset\, W$ is a subbundle
such that $\mu(W)\,=\, \mu(F)$, then it can be shown that $F$ is also pseudo-stable.
\end{remark}

Let $Z$ be a normal projective variety defined over $k$. Fix a point $z_0\, \in\, Z$,
and also fix an ample line bundle on $Z$. 
Consider the category of strongly pseudo-stable vector bundles on $Z$ of degree zero; the morphisms 
from $E$ to $F$ are all ${\mathcal O}_Z$--linear homomorphisms $E\, \longrightarrow\, F$. 
It is a neutral Tannakian category; the fiber functor sends a strongly pseudo-stable vector bundle 
$E\, \longrightarrow\, Z$ of degree zero to its fiber $E_{z_0}$ over $z_0$ \cite[Theorem 1]{BaPa} (see also 
\cite{BPS}). Let $\varpi(Z,\, z_0)$ denote the proalgebraic group scheme over $k$ defined
by this neutral Tannakian category.

\begin{remark}\label{rem1}
In \cite{BaPa}, in the definition of $\varpi(Z,\, z_0)$ the variety $Z$ is assumed to be smooth. However
the construction of $\varpi(Z,\, z_0)$ generalizes to any normal projective variety $Z$. Given a
strongly pseudo-stable vector bundle $E$ on $Z$, and a (smooth) point $z\,\in\, Z$, the restriction of $E$ to
the curves passing through $z$ are considered in the construction of \cite{BaPa}. If $z$ is a singular point
of $Z$, consider the normalization $\widehat{Z}$ of the blow-up of $Z$ at the point $z$. Now instead
of curves on $Z$ passing through $z$ we need to consider curves on $\widehat Z$ passing through a smooth
point of $\widehat{Z}$ lying over $z$.
\end{remark}

\begin{theorem}\label{thm3a}
Take $f\, :\, Y\, \longrightarrow\, X$ as in \eqref{e1}. Fix a point
$y_0\, \in\, Y$ and an ample line bundle on $X$. Then the homomorphism
$$
f_*\,\, :\,\, \varpi(Y,\, y_0)\, \longrightarrow\, \varpi(X,\, f(y_0))
$$
induced by $f$ is faithfully flat (in other words, $f$ is surjective).
\end{theorem}

\begin{proof}
We use the criterion in \cite[p.~139, Proposition 2.21 (a)]{DMOS} for a homomorphism
between two affine group schemes over $k$ to be faithfully flat. From \cite[Lemma 4.3]{BP1}
we know that the first condition in \cite[Proposition 2.21 (a)]{DMOS} is satisfied
for the homomorphism $f_*\, :\, \varpi(Y,\, y_0)\, \longrightarrow\, \varpi(X,\, f(y_0))$. From
Theorem \ref{thm1}(2) it follows immediately that the second condition in
\cite[Proposition 2.21 (a)]{DMOS} is satisfied for $f_*$. Consequently, $f_*$ is faithfully flat
by \cite[p.~139, Proposition 2.21 (a)]{DMOS}.
\end{proof}

\subsection{Surjective generically smooth maps}

As before, let $Z$ be a normal projective variety defined over $k$. Fix a point $z_0\, \in\, Z$.
The category of strongly pseudo-stable vector bundles on $Z$ of degree zero has a full subcategory defined
by the strongly pseudo-stable vector bundles $E\, \longrightarrow\, Z$ such that $c_i(E)$ is numerically
equivalent to zero for all $i\, \geq\, 1$. While the definition of pseudo-stable vector bundles on $Z$ requires
a choice of polarization on $Z$, this subcategory defined by the
strongly pseudo-stable vector bundles $E\, \longrightarrow\, Z$ with $c_i(E)$ numerically
equivalent to zero for all $i\, \geq\, 1$ is actually independent of the choice of polarization on $Z$ (see
\cite{La}, \cite{BH}). Moreover, it is again a neutral Tannakian category. Let 
$\varpi^S(Z,\, z_0)$ denote the proalgebraic group scheme over $k$ defined
by this neutral Tannakian category. Note that $\varpi^S(Z,\, z_0)$ is a quotient of
$\varpi(Z,\, z_0)$.

Let $\beta\,:\, Z\, \longrightarrow\, Z'$ be a morphism of normal projective varieties. This induces
a homomorphism
$$
\beta_*\,:\, \varpi^S(Z,\, z_0)\, \longrightarrow\, \varpi^S(Z',\, \beta(z_0))
$$
which is constructed by sending any strongly pseudo-stable vector bundle $E\, \longrightarrow\, Z$, such that $c_i(E)$ is
numerically equivalent to zero for all $i\, \geq\, 1$, to the pullback $\beta^*E$.

Let $X$ and $Y$ be irreducible normal projective varieties with $\dim Y\, \geq\, \dim X$,
and let
\begin{equation}\label{e1a}
f\,\, :\,\, Y\, \longrightarrow\, X
\end{equation}
be a morphism such that
\begin{itemize}
\item $f$ is surjective,

\item $f$ is generically smooth, and

\item the homomorphism between the \'etale fundamental groups induced by $f$
$$
f_*\,:\,\pi^{\rm et}_{1}(Y) \, \longrightarrow\,\pi^{\rm et}_{1}(X)
$$
is surjective.
\end{itemize}

\begin{proposition}\label{prop3}
Fix a point $y_0\, \in\, Y$. Then the homomorphism
$$
f_*\,\, :\,\, \varpi^S(Y,\, y_0)\, \longrightarrow\, \varpi^S(X,\, f(y_0))
$$
induced by $f$ in \eqref{e1a} is faithfully flat.
\end{proposition}

\begin{proof}
Let $$Y\, \stackrel{\gamma}{\longrightarrow}\, Y_1 \, \stackrel{\delta}{\longrightarrow}\, X$$ be
the Stein factorization of $f$. Both $\gamma$ and $\delta$ are generically smooth because $f$ is so. Let
$$
\gamma_*\,:\, \varpi^S(Y,\, y_0)\, \longrightarrow\, \varpi^S(Y_1,\, \gamma(y_0))\,\ \text{ and }\, \
\delta_*\,:\, \varpi^S(Y_1,\, \gamma(y_0))\, \longrightarrow\, \varpi^S(X,\, f(y_0))
$$
be the homomorphisms induced by $\gamma$ and $\delta$ respectively. Since $\varpi^S(Z,\, z_0)$
is a quotient of $\varpi(Z,\, z_0)$, from Theorem \ref{thm3a} we know that $\delta_*$ is faithfully
flat. Since the fibers of $\gamma$ are connected, the homomorphism $\gamma_*$ is faithfully flat; this follows
using the criterion in \cite[p.~139, Proposition 2.21 (a)]{DMOS} for a homomorphism between two affine
group schemes over $k$ to be faithfully flat. Consequently, the composition of homomorphisms
$$
f_*\,=\, \delta_*\circ\gamma_*
$$
is also faithfully flat.
\end{proof}


\end{document}